\documentclass[12pt]{article}
\usepackage{amsmath,amsthm,amssymb}
\usepackage{amssymb,latexsym}
\newtheorem{theorem}{Theorem}

\newtheorem{lemma}{Lemma}

\begin{document}
\baselineskip=17pt

%

\title{\bf Square-free values of $\mathbf{n^2+n+1}$ }
\author{\bf S. I. Dimitrov}

\date{}
\maketitle

\begin{abstract}

In this paper we show that there exist infinitely many square-free numbers of the form $n^2+n+1$.
We achieve this by deriving an asymptotic formula by improving the reminder term from previous results.
\quad\\
{\bf  2020 Math.\ Subject Classification}:  11L05 $\cdot$ 11N25 $\cdot$  11N37
\end{abstract}

\section{Notations}
\indent

Let $X$ be a sufficiently large positive number.
By $\varepsilon$ we denote an arbitrary small positive number, not necessarily the same in different occurrences.
As usual $\mu(n)$ is M\"{o}bius' function and $\tau(n)$ denotes the number of positive divisors of $n$.
Further $[t]$ and $\{t\}$ denote the integer part, respectively, the fractional part of $t$.
We shall use the convention that a congruence, $m\equiv n\,\pmod {d}$ will be written as $m\equiv n\,(d)$.
As usual $(m,n)$ is the greatest common divisor of $m$ and $n$.
The letter $p$  will always denote prime number.
Moreover $e(t)$=exp($2\pi it$) and $\psi(t)=\{t\}-1/2$.
For $x, y \in\mathbb{R}$ we write $x\equiv y\,(1)$ when $x-y\in\mathbb{Z}$.
For any $n$ and $q$ such that $(n, q)=1$ we denote by $\overline{n}_q$
the inverse of $n$ modulo $q$.
By $K(r,h)$ we shall denote the incomplete Kloosterman sum
\begin{equation}\label{Kloosterman}
K(r,h)=\sum\limits_{\alpha\leq x\leq \beta\atop{(x, r)=1}}e\left(\frac{h\overline{x}_{|r|}}{r}\right)\,,
\end{equation}
where
\begin{equation*}
h, r\in\mathbb{Z}, \quad hr\neq 0, \quad 0<\beta-\alpha\leq2|r|.
\end{equation*}
We also define
\begin{equation}\label{lambdaq1q2}
\lambda(q)=\sum\limits_{1\leq n\leq q\atop{n^2+n+1\equiv 0\,(q)}}1\,,
\end{equation}
\begin{equation}\label{GammaX}
\Gamma(X)=\sum\limits_{1\leq n\leq X}\mu^2(n^2+n+1)\,.
\end{equation}

\section{Introduction and statement of the result}
\indent

We have many reasons to presuppose that there exist infinitely many prime numbers of the form $n^2+n+1$.
Unfortunately, this hypothesis is beyond the scope of modern number theory.
Radically new ideas are needed to be able to attack this problem.
For now, we need to focus on the study of square-free numbers of the shape $n^2+n+1$.
The first result in this direction was obtained in 1933 and belongs to Ricci \cite{Ricci}.
Consider the irreducible polynomial $f(x)\in \mathbb{Z}[x]$ of degree $2$.
Assume that for every prime $p$ there is at least one integer $n_p$ for which $p^2\nmid f(n_p)$.
Ricci proved that the asymptotic formula
\begin{equation}\label{Ricci}
N_f(X)= C_f X + \Delta
\end{equation}
holds.
Here
\begin{equation*}
N_f(X)=\#\{n\leq X: \,f(n) \mbox{ is square-free}\}\,,
\end{equation*}
\begin{equation*}
C_f=\prod\limits_{p}\left(1-\frac{\rho_f(p^2)}{p^2}\right)\,,
\end{equation*}
\begin{equation*}
\rho_f(d) = \#\{a \,\emph{(mod }\, d) : d \,|\, f(a)\}
\end{equation*}
and
\begin{equation*}
\Delta=o(X)\,.
\end{equation*}
Subsequently the reminder term in \eqref{Ricci} was sharpen by Nair \cite{Nair} to
\begin{equation*}
\Delta=\mathcal{O}\left(\frac{X}{\log X}\right)\,.
\end{equation*}
In this paper we improve the reminder term of Nair about the distribution of square-free  polynomials of the shape $n^2+n+1$.
More precisely we establish the following theorem.
\begin{theorem}\label{Theorem1}  For the sum $\Gamma(X)$ defined by \eqref{GammaX} the asymptotic formula
\begin{equation}\label{asymptoticformula1}
\Gamma(X)=\sigma X+\mathcal{O}\left(X^{\frac{4}{5}+\varepsilon}\right)
\end{equation}
holds. Here
\begin{equation}\label{sigmaproduct}
\sigma =\prod_p\left(1-\frac{\lambda(p^2)}{p^2}\right)\,.
\end{equation}
\end{theorem}

A bijection correspondence between the number of representations of number by binary quadratic
form and the incongruent solutions of quadratic congruence allows us to achieve effective control
over the distribution of square-free numbers of the form $n^2+n+1$.
Results related to power-free values of polynomials can be found in
\cite{Browning}, \cite{Erdos}, \cite{Estermann}, \cite{Heath-Brown2}, \cite{Heath-Brown3},
\cite{Heath-Brown4}, \cite{Hooley}, \cite{Reuss2}.

\section{Lemmas}
\indent

This lemma gives us important expansions.
\begin{lemma}\label{expansion}
For any $M\geq2$, we have
\begin{equation*}
\psi(t)=-\sum\limits_{1\leq|m|\leq M}\frac{e(mt)}{2\pi i m}
+\mathcal{O}\big(f_M(t)\big)\,,
\end{equation*}
where $f_M(t)$ is a positive function of $t$ which is infinitely many
times differentiable and periodic with period 1.
It can be expanded into the Fourier series
\begin{equation*}
f_M(t)=\sum\limits_{m=-\infty}^{+\infty}b_{M}(m)e(m t)\,,
\end{equation*}
with coefficients $b_{M}(m)$ such that
\begin{equation*}
b_{M}(m)\ll\frac{\log M}{M}\quad \mbox{for all}\quad m
\end{equation*}
and
\begin{equation*}
\sum\limits_{|m|>M^{1+\varepsilon}}|b_{M}(m)|\ll M^{-A}\,.
\end{equation*}
Here $A > 0$ is arbitrarily large and the constant in the $\ll$ - symbol depends on $A$ and $\varepsilon$.
\end{lemma}
\begin{proof}
See (\cite{Tolev}, Theorem 1).
\end{proof}
The next lemma we need is well-known.
\begin{lemma}\label{Wellknown}
Let $A, B\in\mathbb{Z}\setminus \{0\}$ and $(A, B)=1$. Then
\begin{equation*}
\frac{\overline{A}_{|B|}}{B}+\frac{\overline{B}_{|A|}}{A}\equiv \frac{1}{AB}\,\,(\,1\,).
\end{equation*}
\end{lemma}

\begin{lemma}\label{Weilsestimate}
For the sum denoted by \eqref{Kloosterman} the estimate
\begin{equation*}
K(r,h)\ll|r|^{\frac{1}{2}+\varepsilon}\,(r,h)^{\frac{1}{2}}
\end{equation*}
holds.
\end{lemma}
\begin{proof}
Follows easily from  A. Weil's estimate for the Kloosterman sum.
See (\cite{Iwaniec}, Ch. 11, Corollary 11.12).
\end{proof}

\begin{lemma}\label{Bijection}
Let $n\geq3$ . There exists a bijective function from the solution set of the equation
\begin{equation}\label{Equation}
x^2+xy+y^2=n\,, \quad (x,y)=1\,, \quad x, y \in\mathbb{N}
\end{equation}
to the incongruent solutions modulo $n$ of the congruence
\begin{equation}\label{Congruence}
z^2+z+1\equiv 0\,(n)\,.
\end{equation}
\end{lemma}
\begin{proof}
Let $F$ denote the set of ordered pairs $(x,y)$ satisfying \eqref{Equation}
and $E$ denote the set of solutions of the congruence \eqref{Congruence}.
We consider each residue class modulo $n$ with representatives satisfying \eqref{Congruence}
as one solution of \eqref{Congruence}.

Let $(x,y)\in F$. From \eqref{Equation}  it follows that $(n, y) = 1$.
Therefore there exists a unique residue class $z$ modulo $n$ such that
\begin{equation}\label{zyx}
zy\equiv  x\,(n)\,.
\end{equation}
For this class we have
\begin{equation*}
(z^2+z+1)y^2\equiv(zy)^2+(zy)y+y^2\equiv x^2+xy+y^2\equiv0 \,(n)\,.
\end{equation*}
From the last congruence and $(n, y) = 1$ we obtain
\begin{equation*}
z^2+z+1\equiv 0\,(n)
\end{equation*}
which means that $z\in E$.
We define the map
\begin{equation}\label{map}
\beta : F\rightarrow E
\end{equation}
that associates to each pair  $(x,y)\in F$ the  residue class $z=x\overline{y}_n$ satisfying \eqref{zyx}.\\
We will first prove that the map \eqref{map} is a injection.
Let  $(x,y),\,(x',y')\in F$ that is
\begin{equation}\label{System1}
\left|\begin{array}{cc}
x^2+xy+y^2=n\;\;\,\\
x'^2+x'y'+y'^2=n \end{array}\right.\,,
\end{equation}
\begin{equation}\label{coprimexyx'y'}
(x,y)=(x',y')=1
\end{equation}
and
\begin{equation}\label{xyx'y'}
(x,y)\neq(x',y')\,.
\end{equation}
Assume that
\begin{equation}\label{betaxyx'y'}
\beta(x,y)=\beta(x',y')\,.
\end{equation}
Hence there exists $z\in E$ such that
\begin{equation}\label{System2}
\left|\begin{array}{cc}
zy\equiv  x\,(n)\quad\\
zy'\equiv  x'\,(n)\;\;
\end{array}\right..
\end{equation}
The system \eqref{System2} implies
\begin{equation}\label{xy'x'y}
xy'-x'y\equiv 0\,(n)\,.
\end{equation}
By \eqref{System1} we deduce
\begin{equation*}
0<x, x', y, y'<\sqrt{n}\,.
\end{equation*}
Hence
\begin{equation}\label{-nxy'x'yn}
-n<xy'-x'y<n\,.
\end{equation}
Now \eqref{xy'x'y} and \eqref{-nxy'x'yn} lead to
\begin{equation*}
xy'-x'y=0
\end{equation*}
which together with \eqref{coprimexyx'y'} gives us
\begin{equation}\label{xx'yy'}
x=x'\,, \quad y=y'\,.
\end{equation}
From \eqref{xyx'y'} and \eqref{xx'yy'} we get a contradiction.
Therefore the assumption \eqref{betaxyx'y'} is not true.
This proves the injectivity of $\beta$.

It remains to show that the map \eqref{map} is a surjection. Let $z\in E$.
From Dirichlet's approximation theorem it follows that there exist integers $a$ and $q$ such that
\begin{equation}\label{Dirichlet}
\left|\frac{z}{n}-\frac{a}{q}\right|<\frac{1}{q\sqrt{n}}\,,
  \quad\quad 1\leq q\leq \sqrt{n},\quad\quad (a,\,q)=1\,.
\end{equation}
Put
\begin{equation}\label{r}
r=zq-an.
\end{equation}
Thus
\begin{equation}\label{r^2+2q^2}
r^2+rq+q^2=z^2q^2-2zqan+a^2n^2+zq^2-anq+q^2\equiv(z^2+z+1)q^2\,(n)\,.
\end{equation}
From \eqref{Congruence} and \eqref{r^2+2q^2} we deduce
\begin{equation}\label{r^2+2q^2modn}
r^2+rq+q^2\equiv0\,(n)\,.
\end{equation}
By \eqref{Dirichlet} and \eqref{r} it follows
\begin{equation}\label{rest}
|r|<\sqrt{n}\,.
\end{equation}
Using \eqref{Dirichlet} and \eqref{rest} we get
\begin{equation}\label{r^2+2q^2est}
0<r^2+rq+q^2<3n\,.
\end{equation}
Taking into account \eqref{r^2+2q^2modn} and \eqref{r^2+2q^2est} we conclude that
$r^2+rq+q^2=n$ or $r^2+rq+q^2=2n$.

We consider two cases.

\smallskip

\textbf{Case 1}
\begin{equation}\label{Case1}
r^2+rq+q^2=n\,.
\end{equation}
From \eqref{r} and \eqref{Case1} we obtain
\begin{equation*}
n=(zq-an)^2+(zq-an)q+q^2=(z^2+z+1)q^2-ran-zqan-qan
\end{equation*}
and therefore
\begin{equation}\label{ra+1}
ra+1=kq\,,
\end{equation}
where
\begin{equation}\label{kznqaz}
k=\frac{z^2+z+1}{n}q-az-a\,.
\end{equation}
By \eqref{Congruence} and \eqref{kznqaz} it follows that  $k\in\mathbb{Z}$ and bearing in mind \eqref{ra+1} establish that
\begin{equation}\label{rq1}
(r,q)=1\,.
\end{equation}
Using \eqref{Case1}, \eqref{rq1} and $n\geq3$ we deduce $r\neq0$.

\smallskip

\textbf{Case 1.1}
\begin{equation*}
r>0\,.
\end{equation*}
Put
\begin{equation}\label{xy1}
x=r\,, \quad y=q\,.
\end{equation}
Now \eqref{Case1}, \eqref{rq1} and \eqref{xy1} imply $(x,y)\in F$.
Also \eqref{r} and  \eqref{xy1} yield  \eqref{zyx}. Therefore $\beta(x,y)=z$.

\smallskip

\textbf{Case 1.2}
\begin{equation*}
r<0\,.
\end{equation*}

\textbf{Case 1.2.1}
\begin{equation*}
r+q>0\,.
\end{equation*}
We use that \eqref{Case1} is equivalent to
\begin{equation}\label{r+q>0}
(r+q)^2-(r+q)r+r^2=n\,.
\end{equation}
Set
\begin{equation}\label{xr+qr1}
x=r+q\,, \quad y=-r\,.
\end{equation}
From \eqref{rq1} we have
\begin{equation}\label{r+q,r1}
(r+q, r)=1\,.
\end{equation}
Now \eqref{r+q>0}, \eqref{xr+qr1} and \eqref{r+q,r1} give us $(x,y)\in F$.
As well \eqref{r} and  \eqref{xr+qr1} lead to \eqref{zyx}. Thus $\beta(x,y)=z$.

\smallskip

\textbf{Case 1.2.2}
\begin{equation*}
r+q<0\,.
\end{equation*}
We use that \eqref{Case1} is equivalent to
\begin{equation}\label{r+q<0}
q^2-q(r+q)+(r+q)^2=n\,.
\end{equation}
Put
\begin{equation}\label{xr+qr2}
x=q\,, \quad y=-r-q\,.
\end{equation}
By \eqref{rq1} we have
\begin{equation}\label{r+q,r2}
(q, r+q)=1\,.
\end{equation}
Now \eqref{r+q<0}, \eqref{xr+qr2} and \eqref{r+q,r2} yield $(x,y)\in F$.
As well \eqref{r} and  \eqref{xr+qr2} lead to \eqref{zyx}. Thus $\beta(x,y)=z$.

\smallskip

\textbf{Case 2}
\begin{equation}\label{Case2}
r^2+rq+q^2=2n\,.
\end{equation}
From \eqref{Case2} it follows that
\begin{equation}\label{rqeven}
r=2r_0\,, \quad q=2q_0\,.
\end{equation}
Now \eqref{Case2} and \eqref{rqeven} imply
\begin{equation}\label{r0q0}
2(r_0^2+r_0q_0+q_0^2)=n\,,
\end{equation}
that is, $n$ is an even number that contradicts \eqref{Congruence}.
Consequently this case is impossible.

The lemma is proved.
\end{proof}

\section{Outline of the proof }
\indent

The preliminary manoeuvres for this problem are straightforward.
Using \eqref{GammaX} and the well-known identity
\begin{equation*}
\mu^2(n)=\sum_{d^2|n}\mu(d)
\end{equation*}
we write
\begin{equation}\label{GammaXdecomp}
\Gamma(X)=\sum\limits_{1\leq d\leq\sqrt{X^2+X+1}}\mu(d)\sum\limits_{1\leq n\leq X\atop{n^2+n+1\equiv 0\,(d^2)}}1
=\Gamma_1(X)+\Gamma_2(X)\,,
\end{equation}
where
\begin{align}
\label{Gamma1}
&\Gamma_1(X)=\sum\limits_{1\leq d\leq z}\mu(d)\Sigma\big(X, d^2\big)\,,\\
\label{Gamma2}
&\Gamma_2(X)=\sum\limits_{d>z}\mu(d)\Sigma\big(X, d^2\big)\,,\\
\label{Sigma}
&\Sigma\big(X, d^2\big)=\sum\limits_{1\leq n\leq X\atop{n^2+n+1\equiv 0\,(d^2)}}1\,,\\
\label{z}
&\sqrt{X}\leq z< X\,,
\end{align}
where $z$ is to be chosen later.

We shall estimate $\Gamma_1$ and $\Gamma_2$, respectively,
in the sections \ref{SectionGamma1} and \ref{SectionGamma2}.
In section \ref{Sectionfinal} we shall finalize the proof of Theorem \ref{Theorem1}.

\section{Estimation of $\mathbf{\Gamma_1(X)}$}\label{SectionGamma1}
\indent

Hencefort we assume that $q=d^2$, where $d$  is square-free and $d\leq z$.

Denote
\begin{equation}\label{Omega}
\Omega(X, q, n)=\sum\limits_{m\leq X\atop{m\equiv n\,(q)}}1\,.
\end{equation}
Obviously
\begin{equation}\label{Omegaest}
\Omega(X, q, n)=\frac{X}{q}+\mathcal{O}(1)\,.
\end{equation}
Using \eqref{Sigma} and \eqref{Omega} we obtain upon partitioning
the sum \eqref{Sigma} into residue classes modulo $q$
\begin{equation}\label{SigmaOmega}
\Sigma\big(X, q\big)=\sum\limits_{1\leq n\leq q\atop{n^2+n+1\equiv 0\,(q)}}\Omega(X, q, n)\,.
\end{equation}
By  \eqref{lambdaq1q2}, \eqref{SigmaOmega} and \eqref{Omegaest} we obtain
\begin{equation}\label{Sigmaest1}
\Sigma\big(X, q\big)=X\frac{\lambda(q)}{q}+\mathcal{O}\big(\lambda(q)\big)\,.
\end{equation}
Bearing in mind \eqref{lambdaq1q2} and that the number of solutions of the congruence
\begin{equation*}
n^2+n+1\equiv a\,(q)
\end{equation*}
is less than or equal to $\tau(q)$ we find
\begin{equation}\label{lambdaq1q2est}
\lambda(q)\ll\tau(q)\,.
\end{equation}
Now \eqref{Sigmaest1}, \eqref{lambdaq1q2est} and the inequalities
\begin{equation*}
\tau(q)\ll (q)^\varepsilon\ll X^\varepsilon
\end{equation*}
yield
\begin{equation}\label{Sigmaest2}
\Sigma\big(X, q\big)=X\frac{\lambda(q)}{q}+\mathcal{O}\big(X^\varepsilon\big)\,.
\end{equation}
Taking into account \eqref{Gamma1}, \eqref{z} and \eqref{Sigmaest2} we deduce
\begin{align}\label{GammaX1est1}
\Gamma_1(X)&=X\sum\limits_{1\leq d\leq z}\frac{\mu(d)\lambda(d^2)}{d^2}+\mathcal{O}\big(zX^\varepsilon \big)\nonumber\\
&=\sigma X-X \sum\limits_{d>z}\frac{\mu(d)\lambda(d^2)}{d^2}+\mathcal{O}\big(zX^\varepsilon\big)\,,
\end{align}
where
\begin{equation}\label{sigmasum}
\sigma=\sum\limits_{d=1}^\infty\frac{\mu(d)\lambda(d^2)}{d^2}\,.
\end{equation}
Using \eqref{lambdaq1q2est} we get
\begin{equation}\label{d1d2>est}
\sum\limits_{d>z}\frac{\mu(d)\lambda(d^2)}{d^2}\ll\sum\limits_{d>z}\frac{d^{\varepsilon}}{d^2}\ll z^{\varepsilon-1}\,.
\end{equation}
It remains to see that the product  \eqref{sigmaproduct} and the sum \eqref{sigmasum} coincide.
From the definition \eqref{lambdaq1q2} it follows that the function $\lambda(q)$ is multiplicative, i.e. if
\begin{equation*}
(q_1, q_2)=1
\end{equation*}
then
\begin{equation*}
\lambda(q_1q_2)=\lambda(q_1)\lambda(q_2)\,.
\end{equation*}
Obviously the function
\begin{equation*}
\frac{\mu(d)\lambda(d^2)}{d^2}
\end{equation*}
is multiplicative  and the series
\begin{equation*}
\sum\limits_{d=1}^\infty\frac{\mu(d)\lambda(d^2)}{d^2}
\end{equation*}
is absolutely convergent.
Applying the Euler product we get
\begin{equation}\label{sigmasumest}
\sigma=\sum\limits_{d=1}^\infty\frac{\mu(d)\lambda(d^2)}{d^2}=
\prod_p\left(1-\frac{\lambda(p^2)}{p^2}\right)\,.
\end{equation}
Bearing in mind \eqref{z}, \eqref{GammaX1est1}, \eqref{d1d2>est} and \eqref{sigmasumest} we obtain
\begin{equation}\label{GammaX1est2}
\Gamma_1(X)=\sigma X+\mathcal{O}\big(zX^\varepsilon\big)\,,
\end{equation}
where $\sigma$ is given by the product \eqref{sigmaproduct}.

\section{Estimation of $\mathbf{\Gamma_2(X)}$}\label{SectionGamma2}
\indent

Using \eqref{Gamma2}, \eqref{Sigma} and splitting the range of $d$ into dyadic subintervals of the form $D\leq d<2D$,
we write
\begin{equation}\label{Gamma2est1}
\Gamma_2(X)\ll(\log X)\Sigma_0\,,
\end{equation}
where
\begin{equation}\label{Sigma0}
\Sigma_0=\sum\limits_{n\leq X}\sum\limits_{D\leq d<2D\atop{n^2+n+1\equiv 0\,(d^2)}}1\,,
\end{equation}
\begin{equation}\label{zDX}
\frac{z}{2}\leq D\leq\sqrt{X^2+X+1}\,.
\end{equation}
Define
\begin{align}
\label{N1set}
&\mathcal{N}(d)=\{n\in\mathbb{N }\; : \; 1\leq n \leq d, \;\;  n^2+n+1\equiv 0\,(d) \}\,,\\
\label{N'1set}
&\mathcal{N}'(d)=\{n\in\mathbb{N }\; : \; 1\leq n \leq d^2, \;\;  n^2+n+1\equiv 0\,(d^2) \}\,.
\end{align}
By \eqref{Sigma0} and \eqref{N'1set} we obtain
\begin{align}\label{Sigma0est}
\Sigma_0&=\sum\limits_{D\leq d<2D}\sum\limits_{n\in \mathcal{N}'(d)}
\sum\limits_{m\leq X\atop{m\equiv n\,(d^2)}}1
=\sum\limits_{D\leq d<2D}\sum\limits_{n\in \mathcal{N}'(d)}
\Bigg(\left[\frac{X-n}{d^2}\right]-\left[\frac{-n}{d^2}\right]\Bigg) \nonumber\\
&=\sum\limits_{D\leq d<2D}\sum\limits_{n\in \mathcal{N}'(d)}
\Bigg(\frac{X}{d^2}+\psi\left(\frac{-n}{d^2}\right)-\psi\left(\frac{X-n}{d^2}\right)\Bigg)\nonumber\\
&\ll X^{1+\varepsilon}D^{-1}+|\Sigma_1|+|\Sigma_2|\,,
\end{align}
where
\begin{align}
\label{Sigma1}
&\Sigma_1=\sum\limits_{D\leq d<2D}\sum\limits_{n\in \mathcal{N}'(d)}\psi\left(\frac{-n}{d^2}\right)\,,\\
\label{Sigma2}
&\Sigma_2=\sum\limits_{D\leq d<2D}\sum\limits_{n\in \mathcal{N}'(d)}\psi\left(\frac{X-n}{d^2}\right)\,.
\end{align}
Firstly we consider the sum $\Sigma_1$.
We note that the sum over $n$ in  \eqref{Sigma1} does not contain terms with $n=\frac{d^2}{2}$ and $n=d^2$.
Moreover for any $n$ satisfying the congruence $n^2+n+1\equiv 0\,(d^2)$  and such that
$1\leq n<\frac{d^2}{2}$ the number $d^2-n-1$ satisfies the same congruence.
Taking into account these arguments we get
\begin{align}\label{Sigma1est}
\Sigma_1&=\sum\limits_{D\leq d<2D}\sum\limits_{1\leq n<d^2/2\atop{n^2+n+1\equiv 0\,(d^2)}}\psi\left(\frac{-n}{d^2}\right)
+\sum\limits_{D\leq d<2D}\sum\limits_{d^2/2< n<d^2\atop{n^2+n+1\equiv 0\,(d^2)}}\psi\left(\frac{-n}{d^2}\right)\nonumber\\
&=\sum\limits_{D\leq d<2D}\sum\limits_{1\leq n<d^2/2\atop{n^2+n+1\equiv 0\,(d^2)}}\psi\left(\frac{-n}{d^2}\right)
+\sum\limits_{D\leq d<2D}\sum\limits_{1\leq m<d^2/2-1\atop{(d^2-m-1)^2+d^2-m\equiv 0\,(d^2)}}\psi\left(\frac{-(d^2-m-1)}{d^2}\right)\nonumber\\
&=\sum\limits_{D\leq d<2D}\sum\limits_{1\leq n<d^2/2\atop{n^2+n+1\equiv 0\,(d^2)}}\psi\left(\frac{-n}{d^2}\right)
+\sum\limits_{D\leq d<2D}\sum\limits_{1\leq m<d^2/2-1\atop{m^2+m+1\equiv 0\,(d^2)}}\psi\left(\frac{m+1}{d^2}\right)\nonumber\\
&=\sum\limits_{D\leq d<2D}\sum\limits_{1\leq n<d^2/2-1\atop{n^2+n+1\equiv 0\,(d^2)}}
\Bigg(\psi\left(\frac{-n}{d^2}\right)+\psi\left(\frac{n+1}{d^2}\right)\Bigg)
+\sum\limits_{D\leq d<2D}\sum\limits_{d^2/2-1\leq n<d^2/2\atop{n^2+n+1\equiv 0\,(d^2)}}\psi\left(\frac{-n}{d^2}\right)\nonumber\\
&=\sum\limits_{D\leq d<2D}\sum\limits_{1\leq n<d^2/2-1\atop{n^2+n+1\equiv 0\,(d^2)}}\frac{1}{d^2}
+\sum\limits_{D\leq d<2D}\sum\limits_{d^2/2-1\leq n<d^2/2\atop{n^2+n+1\equiv 0\,(d^2)}}\left(\frac{1}{2}-\frac{n}{d^2}\right)\nonumber\\
&\ll X^\varepsilon D^{-1}+\sum\limits_{D\leq d<2D}\sum\limits_{d^2/2-1\leq n<d^2/2\atop{n^2+n+1\equiv 0\,(d^2)}}\frac{1}{d^2}\nonumber\\
&\ll X^\varepsilon D^{-1}\,.
\end{align}
Next we consider the sum $\Sigma_2$ denoted by \eqref{Sigma2}.
If $D\leq X^{\frac{1}{2}}$  then trivial estimation gives us
\begin{equation}\label{Sigma2est1}
\Sigma_2\ll\sum\limits_{D\leq d<2D}d^\varepsilon\ll  X^{\frac{1}{2}+\varepsilon}\,.
\end{equation}
Assume
\begin{equation}\label{D>}
D> X^{\frac{1}{2}}\,.
\end{equation}
First we notice that all summands in the sum  \eqref{Sigma2} for which $3\mid d$  are equal to zero
because the congruences
\begin{equation}\label{congruence}
n^2+n+1\equiv 0\,(9)
\end{equation}
has no solution. For this reason, in the estimation of \eqref{Sigma2} we will consider that $3\nmid d$.\\
Let $f(x) = a_nx^n + a_{n-1}x^{n-1} + \cdots + a_0$ be a polynomial with integral coefficients
and $r_1, \ldots, r_k$ be all solutions of the congruence
\begin{equation}\label{congruence1}
f(x)\equiv0\,(p^{l-1})\,.
\end{equation}
From the theory of the congruences we know that when $p\nmid f'(r_i)$ for $i=1,\ldots, k$ then
the number of solutions of the congruence
\begin{equation}\label{congruence2}
f(x)\equiv0\,(p^l)\,.
\end{equation}
is also equal to $k$, that is, congruences \eqref{congruence1} and \eqref{congruence2} have an equal number of solutions.
Given the above considerations, we conclude that the congruences
\begin{equation}\label{congruence3}
n^2+n+1\equiv 0\,(d^2)
\end{equation}
and
\begin{equation}\label{congruence4}
n^2+n+1\equiv 0\,(d)
\end{equation}
will have an equal number of solutions if we prove that for arbitrary prime factor $p$ of $d$ and
arbitrary solution $r$ of \eqref{congruence4} we have that
\begin{equation*}
p\nmid 2r+1\,.
\end{equation*}
We assume the opposite. Then
\begin{equation*}
p\mid r^2+r+1 \quad \mbox{ and } \quad p\mid2r+1
\end{equation*}
and therefore
\begin{equation*}
p\mid r-1
\end{equation*}
that is $r=ph+1$, where $h\in\mathbb{Z}$. Now
\begin{equation*}
r^2+r+1\equiv 0\,(p)
\end{equation*}
leads to
\begin{equation*}
(ph+1)^2+ph+1+1=p^2h^2+3ph+3\equiv 0\,(p)
\end{equation*}
which means $p=3$. But we have already excluded the case when 3 is a prime factor of $d$.
Consequently congruences \eqref{congruence3} and \eqref{congruence4} have an equal number of solutions.
Moreover for any $n$ satisfying the congruence \eqref{congruence4}  and such that
$1\leq n<\frac{d^2}{2}$ the number $d^2-n-1$ satisfies the same congruence.
As we mentioned the same is true for the congruence \eqref{congruence3}.
We also note that if $n=d^2-n-1$ then $d^2 \mid3$ which is impossible
and if $n=d-n-1$ then $d=3$ which we excluded  as an possibility.
Using this fact and notations \eqref{N1set}, \eqref{N'1set} we denote
\begin{equation}\label{k}
k=\#\mathcal{N}(d)=\#\mathcal{N}'(d)\,,
\end{equation}
\begin{equation}\label{solutions}
n_1, \ldots, n_k\in\mathcal{N}(d)\,,\quad n'_1, \ldots, n'_k\in\mathcal{N}'(d)\,.
\end{equation}
Now  \eqref{N1set}, \eqref{N'1set}, \eqref{D>}, \eqref{k} and \eqref{solutions}  imply
\begin{align}\label{Prehod}
&\sum\limits_{n\in \mathcal{N}'(d)}\psi\left(\frac{X-n}{d^2}\right)
=\sum\limits_{n\in \mathcal{N}'(d)}\left(\frac{X-n}{d^2}-\frac{1}{2}\right)\nonumber\\
&=\sum\limits_{n\in \mathcal{N}'(d)}\left(\frac{X}{d^2}-\frac{1}{2}\right)
-\frac{n'_1+\cdots+n'_{k/2}+(d^2-n'_1-1)+\cdots+(d^2-n'_{k/2}-1)}{d^2}\nonumber\\
&=\sum\limits_{n\in \mathcal{N}(d)}\left(\frac{X}{d^2}-\frac{1}{2}\right)
-\frac{k(d^2-1)}{2d^2}\nonumber\\
&=\sum\limits_{n\in \mathcal{N}(d)}\left(\frac{X}{d^2}-\frac{1}{2}\right)-
\frac{n_1+\cdots+n_{k/2}+(d-n_1-1)+\cdots+(d-n_{k/2}-1)}{d}\left(1+\frac{1}{d}\right)\nonumber\\
&=\sum\limits_{n\in \mathcal{N}(d)}\left(\frac{X}{d^2}-\frac{1}{2}\right)
-\left(1+\frac{1}{d}\right)\sum\limits_{n\in \mathcal{N}(d)}\frac{n}{d}\nonumber\\
&=\sum\limits_{n\in \mathcal{N}(d)}\left(\frac{X}{d^2}-\frac{\sqrt{X}}{d}-\frac{n}{d^2}\right)
+\sum\limits_{n\in \mathcal{N}(d)}\left(\frac{\sqrt{X}-n}{d}-\frac{1}{2}\right)\nonumber\\
&=\sum\limits_{n\in \mathcal{N}(d)}\left(\frac{X}{d^2}-\frac{\sqrt{X}}{d}-\frac{n}{d^2}\right)
+\sum\limits_{n\in \mathcal{N}(d)}\psi\left(\frac{\sqrt{X}-n}{d}\right).
\end{align}
By \eqref{Sigma2}, \eqref{D>}  and \eqref{Prehod} we obtain
\begin{equation}\label{Sigma2est2}
\Sigma_2\ll  X^{\frac{1}{2}+\varepsilon}+|\Sigma_3|\,,
\end{equation}
where
\begin{equation}\label{Sigma3}
\Sigma_3=\sum\limits_{D\leq d<2D}\sum\limits_{n\in \mathcal{N}(d)}\psi\left(\frac{\sqrt{X}-n}{d}\right)\,.
\end{equation}
Using  \eqref{Sigma3} and Lemma \ref{expansion} with
\begin{equation}\label{M}
M=X^{\frac{1}{2}}
\end{equation}
we deduce
\begin{align}\label{Sigma3est1}
\Sigma_3&=\sum\limits_{D\leq d<2D}\sum\limits_{n\in \mathcal{N}(d)}
\Bigg(-\sum\limits_{1\leq|m|\leq M}\frac{e\left(m\left(\frac{\sqrt{X}-n}{d}\right)\right)}{2\pi i m}
+\mathcal{O}\left(f_{M}\left(\frac{\sqrt{X}-n}{d}\right)\right)\Bigg)\nonumber\\
&=\Sigma_4+\Sigma_5\,,
\end{align}
where
\begin{align}
\label{Sigma4}
&\Sigma_4=\sum\limits_{1\leq|m|\leq M}\frac{\Theta_m}{2\pi i m}\,,\\
\label{Thetam}
&\Theta_m=\sum\limits_{D\leq d<2D}e\left(\frac{\sqrt{X}m}{d}\right)
\sum\limits_{n\in \mathcal{N}(d)}e\left(-\frac{nm}{d}\right)\,,\\
\label{Sigma5}
&\Sigma_5=\sum\limits_{D\leq d<2D}\sum\limits_{n\in \mathcal{N}(d)}f_{M}\left(\frac{\sqrt{X}-n}{d}\right)\,.
\end{align}
Now \eqref{Thetam}, \eqref{Sigma5} and Lemma \ref{expansion} yield
\begin{align}\label{Sigma5est}
\Sigma_5&=\sum\limits_{D\leq d<2D}\sum\limits_{n\in \mathcal{N}(d)}
\sum\limits_{m=-\infty}^{+\infty}b_{M}(m)e\left(\frac{\sqrt{X}-n}{d}m\right)
=\sum\limits_{m=-\infty}^{+\infty}b_{M}(m)\Theta_m\nonumber\\
&\ll\frac{\log M}{M}|\Theta_0|+\frac{\log M}{M}\sum\limits_{1\leq|m|\leq M^{1+\varepsilon}}|\Theta_m|
+\sum\limits_{|m|> M^{1+\varepsilon}}|b_{M}(m)||\Theta_m|\nonumber\\
&\ll\frac{\log M}{M}D^{1+\varepsilon}+\frac{\log M}{M}\sum\limits_{1\leq m\leq M^{1+\varepsilon}}|\Theta_m|
+D^{1+\varepsilon}\sum\limits_{|m|> M^{1+\varepsilon}}|b_{M}(m)|\nonumber\\
&\ll\frac{\log M}{M}D^{1+\varepsilon}+\frac{\log M}{M}\sum\limits_{1\leq m\leq M^{1+\varepsilon}}|\Theta_m|\,.
\end{align}
From \eqref{Sigma3est1}, \eqref{Sigma4} and \eqref{Sigma5est} it follows
\begin{equation}\label{Sigma3est2}
\Sigma_3\ll X^{\varepsilon}\left(\frac{D}{M}+\sum\limits_{1\leq m\leq M^{1+\varepsilon}}\frac{|\Theta_m|}{m}\right).
\end{equation}
Define
\begin{equation}\label{Fdset}
\mathcal{F}(d)=\{(u,v)\; : \; u^2+uv+v^2=d, \;\;  (u,v)=1, \;\; u, v\in\mathbb{N} \}.
\end{equation}
According to Lemma \ref{Bijection}  there exists a bijection
\begin{equation*}
\beta : \mathcal{F}(d)\rightarrow \mathcal{N}(d)
\end{equation*}
from $\mathcal{F}(d)$ to  $\mathcal{N}(d)$ defined by \eqref{N1set}
that associates to each couple $(u,v)\in \mathcal{F}(d)$ the element $n\in\mathcal{N}(d)$  satisfying
\begin{equation}\label{nvu}
nv\equiv  u\,(d)\,.
\end{equation}
Now  \eqref{nvu} implies
\begin{equation*}
n_{u,v}\equiv  u\overline{v}_{d}\,(d)
\end{equation*}
and therefore
\begin{equation}\label{nuvd}
\frac{n_{u,v}}{d}\equiv  u\frac{\overline{v}_{u^2+uv+v^2}}{u^2+uv+v^2}\,\,(1) \,.
\end{equation}
Bearing in mind \eqref{nuvd} and  Lemma \ref{Wellknown}  we find
\begin{align}
\label{nuvd1}
&\frac{n_{u,v}}{d}\equiv \frac{u}{v(u^2+uv+v^2)}-\frac{\overline{u}_{v}}{v}\,\,(1)\,,\\
\label{nuvd2}
&\frac{n_{u,v}}{d}\equiv -\frac{u+v}{u(u^2+uv+v^2)}+\frac{\overline{v}_{u}}{u}\,\,(1)\,.
\end{align} From \eqref{Thetam}, \eqref{Fdset}, \eqref{nuvd1} and \eqref{nuvd2} we deduce
\begin{align}\label{Thetamest1}
\Theta_m&=\sum\limits_{D\leq d<2D}e\left(\frac{m\sqrt{X}}{d}\right)
\sum\limits_{(u,v)\in\mathcal{F}(d)}e\left(-\frac{n_{u,v}}{d}m\right)\nonumber\\
&=\sum\limits_{D\leq d<2D}e\left(\frac{m\sqrt{X}}{d}\right)
\sum\limits_{(u,v)\in\mathcal{F}(d)\atop{0<u<v}}e\left(-\frac{mu}{v(u^2+uv+v^2)}+\frac{m\overline{u}_{v}}{v}\right)\nonumber\\
&+\sum\limits_{D\leq d<2D}e\left(\frac{m\sqrt{X}}{d}\right)
\sum\limits_{(u,v)\in\mathcal{F}(d)\atop{0<v<u}}e\left(\frac{m(u+v)}{u(u^2+uv+v^2)}-\frac{m\overline{v}_{u}}{u}\right)\nonumber\\
&=\sum\limits_{D\leq u^2+uv+v^2<2D\atop{0<u<v\atop{(u, v)=1}}}
e\left(\frac{m\sqrt{X}}{u^2+uv+v^2}-\frac{mu}{v(u^2+uv+v^2)}+\frac{m\overline{u}_{v}}{v}\right)\nonumber\\
&+\sum\limits_{D\leq u^2+uv+v^2<2D\atop{0<v<u\atop{(u, v)=1}}}
e\left(\frac{m\sqrt{X}}{u^2+uv+v^2}+\frac{m(u+v)}{u(u^2+uv+v^2)}-\frac{m\overline{v}_{u}}{u}\right)\nonumber\\
&=\Theta'_m+\Theta''_m\,,
\end{align}
say.
Consider $\Theta'_m$.
Let for any fixed $\sqrt{\frac{D}{3}}\leq v<\sqrt{2D}$ the interval $\big[\eta_1(v), \eta_2(v)\big]$
is a solution with respect to $u$ of the system
\begin{equation}\label{Systemlast}
\left|\begin{array}{ccc}
u^2+uv+v^2<2D\\
u^2+uv+v^2\geq D \;\;\\
0<u<v\quad\quad\quad\;\\
\end{array}\right..
\end{equation}
Denote
\begin{equation}\label{gu}
g(u)=e\left(\frac{m\sqrt{X}}{u^2+uv+v^2}-\frac{mu}{v(u^2+uv+v^2)}\right)\,,
\end{equation}
\begin{equation}\label{Kloostermanvm}
K_{v,m}(t)=\sum\limits_{\eta_1(v)\leq u\leq t\atop{(u, v)=1}}e\left(\frac{m\overline{u}_{v}}{v}\right)\,.
\end{equation}
Using \eqref{Thetamest1} -- \eqref{Kloostermanvm} and Abel's summation formula we obtain
\begin{align}\label{Thetam'est1}
\Theta'_m&=\sum\limits_{\sqrt{\frac{D}{3}}\leq v<\sqrt{2D}}
\sum\limits_{\eta_1(v)\leq u\leq \eta_2(v)\atop{(u, v)=1}}
g(u)e\left(\frac{m\overline{u}_{v}}{v}\right)\nonumber\\
&=\sum\limits_{\sqrt{\frac{D}{3}}\leq v<\sqrt{2D}}\left( g\big(\eta_2(v)\big)K_{v,m}\big(\eta_2(v)\big)
- \int\limits_{\eta_1(v)}^{\eta_2(v)}K_{v,m}(t)\left(\frac{d}{dt}g(t)\right)\,dt \right)\nonumber\\
&\ll\sum\limits_{\sqrt{\frac{D}{3}}\leq v<\sqrt{2D}}\left(1+\frac{m\sqrt{X}}{v^2}\right)
\max_{\eta_1(v)\leq t\leq \eta_2(v)}|K_{v,m}(t)|\,.
\end{align}
We are now in a good position to apply Lemma \ref{Weilsestimate}
because the sum defined by \eqref{Kloostermanvm} is  incomplete Kloosterman sum.
Thus
\begin{equation}\label{Kloostermanvmest}
K_{v,m}(t)\ll v^{\frac{1}{2}+\varepsilon}\,(v,m)^{\frac{1}{2}}\,.
\end{equation}
By \eqref{Thetam'est1} and  \eqref{Kloostermanvmest} we  deduce
\begin{align}\label{Thetam'est2}
\Theta'_m&\ll\sum\limits_{\sqrt{\frac{D}{3}}\leq v<\sqrt{D}}
\left(1+\frac{m\sqrt{X}}{v^2}\right)v^{\frac{1}{2}+\varepsilon}\,(v,m)^{\frac{1}{2}}\nonumber\\
&\ll X^\varepsilon\Big(D^{\frac{1}{4}}+ mX^{\frac{1}{2}}D^{-\frac{3}{4}} \Big)
\sum\limits_{0<v<\sqrt{D}}(v,m)^{\frac{1}{2}}\,.
\end{align}
On the other hand
\begin{equation}\label{sumvm}
\sum\limits_{0<v<\sqrt{D}}(v,m)^{\frac{1}{2}}
\leq\sum\limits_{l|m}l^{\frac{1}{2}}\sum\limits_{v\leq \sqrt{D}\atop{v\equiv0\,(l)}}1
\ll D^{\frac{1}{2}}\sum\limits_{l|m}l^{-\frac{1}{2}}\ll D^{\frac{1}{2}}\tau(m)\ll X^\varepsilon D^{\frac{1}{2}}\,.
\end{equation}
The estimations \eqref{Thetam'est2} and \eqref{sumvm}  yield
\begin{equation}\label{Thetam'est3}
\Theta'_m\ll X^\varepsilon\Big(D^{\frac{3}{4}}+ mX^{\frac{1}{2}}D^{-\frac{1}{4}} \Big)\,.
\end{equation}
Arguing in a similar way for $\Theta''_m$ from \eqref{Thetamest1} we get
\begin{equation}\label{Thetam''est}
\Theta''_m\ll X^\varepsilon\Big(D^{\frac{3}{4}}+ mX^{\frac{1}{2}}D^{-\frac{1}{4}} \Big)\,.
\end{equation}
From \eqref{Thetamest1}, \eqref{Thetam'est3} and \eqref{Thetam''est} it follows
\begin{equation}\label{Thetamest2}
\Theta_m\ll X^\varepsilon\Big(D^{\frac{3}{4}}+ mX^{\frac{1}{2}}D^{-\frac{1}{4}} \Big)\,.
\end{equation}
Now \eqref{Sigma3est2} and \eqref{Thetamest2} give us
\begin{equation}\label{Sigma3est3}
\Sigma_3\ll  X^\varepsilon\Big(DM^{-1}+D^{\frac{3}{4}}+X^{\frac{1}{2}}MD^{-\frac{1}{4}}\Big)\,.
\end{equation}
Bearing in mind \eqref{zDX}, \eqref{M} and \eqref{Sigma3est3} we find
\begin{equation}\label{Sigma3est4}
\Sigma_3\ll  X^{1+\varepsilon}D^{-\frac{1}{4}}\,.
\end{equation}
Using \eqref{zDX}, \eqref{Sigma2est1}, \eqref{Sigma2est2} and \eqref{Sigma3est4} we obtain
\begin{equation}\label{Sigma2est3}
\Sigma_2\ll  X^{1+\varepsilon}D^{-\frac{1}{4}}\,.
\end{equation}
Summarizing \eqref{Gamma2est1}, \eqref{zDX}, \eqref{Sigma0est}, \eqref{Sigma1est} and  \eqref{Sigma2est3} we get
\begin{equation}\label{Gamma2est2}
\Gamma_2(X)\ll  X^{1+\varepsilon}z^{-\frac{1}{4}}\,.
\end{equation}

\section{The end of the proof}\label{Sectionfinal}
\indent

Taking int account \eqref{GammaXdecomp}, \eqref{GammaX1est2}, \eqref{Gamma2est2}
and choosing $z=X^{\frac{4}{5}}$ we establish the asymptotic formula
\eqref{asymptoticformula1}. This completes the proof of Theorem \ref{Theorem1}.

\vskip20pt
\footnotesize
\begin{flushleft}
S. I. Dimitrov\\
Faculty of Applied Mathematics and Informatics\\
Technical University of Sofia \\
Blvd. St.Kliment Ohridski 8, \\
Sofia 1756, Bulgaria\\
e-mail: sdimitrov@tu-sofia.bg\\
\end{flushleft}

\end{document}